\documentclass[12pt]{article}

\usepackage[numbers]{natbib}
\usepackage{amsthm}
\usepackage[english]{babel}
\usepackage[utf8]{inputenc}
\usepackage{amsmath}
\usepackage{amsthm}
\usepackage{graphicx}
\usepackage[colorinlistoftodos]{todonotes}
\usepackage{wasysym}
\usepackage{amsfonts}
\usepackage{ amssymb }

\title{ On the edge metric dimension for the random graph}

\author{Nina Zubrilina}

\date{\today}
\DeclareMathOperator\edim{edim}

\begin{document}
\maketitle
\newtheorem{theorem}{Theorem}[section]
\newtheorem{corollary}{Corollary}[theorem]
\newtheorem{lemma}[theorem]{Lemma}
\newtheorem{defn}{Definition}[section]
\newtheorem{remark}{Remark}[section]
\newtheorem*{claim}{Claim}
\newtheorem{prob}{Problem}[section]
\renewcommand{\theprob}{\arabic{prob}}
\newcommand\underrel[2]{\mathrel{\mathop{#2}\limits_{#1}}}
\newcommand\tab[1][0.5cm]{\hspace*{#1}}

\newcommand{\Prob}{\mathbb{P}}
\newcommand{\E}{\mathbb{E}}
\newcommand{\Q}{\mathbb{Q}}
\newcommand{\R}{\mathbb{R}}
\newcommand{\N}{\mathbb{N}}
\newcommand{\Z}{\mathbb{Z}}
\newcommand{\eps}{\varepsilon}
\newcommand{\norm}[1]{\left\lVert#1\right\rVert}
\newcommand\abs[1]{\left|#1\right|}

\begin{abstract}
Let $G(V, E)$ be a connected simple undirected graph. In this paper we prove that the edge metric dimension (introduced by Kelenc, Tratnik and Yero) of the Erd\H{o}s-R\'enyi random graph $G(n, p)$ is given by:
$$\edim(G(n, p)) = (1 + o(1))\frac{4\log(n)}{\log(1/q)},$$ where $q = 1 - 2p(1-p)^2(2-p)$.
\end{abstract}

\section{Introduction}
  
In \cite{bollobas2012metric},  Bollob{\'a}s, Mitsche and Pralat computed an asymptotic for $\dim(G(n, p))$ for a wide range of probabilities $p(n)$ as a function of $n$. For instance, for constant $p \in (0, 1)$, it was shown that 
$$\dim(G(n, p)) = (1 + o(1))\frac{2\log(n)}{\log(1/Q)},$$ where $Q = p^2 + (1-p)^2$. In this paper we generalize the methods and calculations made by Bollob{\'a}s, Mitsche and Pralat  in \cite{bollobas2012metric} to give an asymptotic for $\edim(G(n, p))$, which is a similar concept introduced by Kelenc, Tratnik and Yero in \cite{2016arXiv160200291K}. Namely, we show that $$\edim(G(n, p)) = (1 + o(1))\frac{4\log(n)}{\log(1/q)}, $$ where $q = 1 - 2p(1-p)^2(2-p)$.

Metric dimension was introduced by Slater in \cite{slater1975leaves} in 1975. The same idea was introduced independently by Harary and Melter in \cite{harary1976metric} a year later. Graphs with $\dim(G) = 1$ and $2$ were characterized in \cite{khuller1996landmarks}, and graphs with $\dim(G) = n-1 $ and $n-2$ were described in \cite{chartrand2000resolvability}. This graph invariant is useful in areas like robot navigation (\cite{khuller1996landmarks}), image processing (\cite{melter1984metric}), and chemistry (\cite{chartrand2000resolvability}, \cite{chartrand2000resolvability1}, \cite{johnson1993structure}) . 

The concept of edge metric dimension was introduced by Kelenc, Tratnik and Yero in \cite{2016arXiv160200291K} in 2016. They computed the edge metric dimension of a range of families of graphs and showed $\edim(G)$ can be less, equal to, or more than $\dim(G)$. They also showed computing $\edim(G)$ is NP-hard in general.

\section{Definitions}
Let $G(V, E)$ be a finite, simple, undirected graph. For $x, y \in E$, we define the distance $d(x, y)$ to be the length of the shortest path between $x$ and $y$. For $e = xy \in E$, $v \in V$, we define $d(e, v) = \min \{ d(x, v), d(y, v) \}$. 

Let $R = \{r_1, \ldots, r_{|R|}\} \subseteq V$. Let $d_R: V \cup E  \to \N^{|R|}$ via $(d_R(x))_i := d(x, r_i)$. We say $R$ distinguishes $v_1, v_2 \in V$ if $d_R(v_1) \neq d_R(v_2)$, and similarly that $R$ distinguishes  $e_1, e_2 \in E$ if $d_R(e_1) \neq d_R(e_2)$.

 A set $R \subseteq V$ is a $\textit{generating set}$ of a graph $G(V, E)$ if, for any two vertices $v_1$ and $v_2$, $d_R(v_1) \neq d_R(v_2)$. A generating set with the smallest number of elements is called a $\textit{basis}$ of $G$, and the number of elements in a basis is the $\emph{dimension}$ of $G$ (denoted $\dim(G)$). 

A set $R \subseteq V$ is an $\textit{edge generating set}$ of a graph $G(V, E)$ if for any two edges $e_1$ and $e_2$, $d_R(e_1) \neq d_R(e_2)$. An edge generating set with the smallest number of elements is called an $\textit{edge basis}$ of $G$, and the number of elements in the edge basis is the $\emph{edge dimension}$ of $G$ (denoted $\edim(G)$).

We say $f(n) = \mathcal{O}(g(n))$ if there exists a constant $C>0$ such that $\abs{f(n)} \leq C \abs{g(n)}$, and $f(n) = o(g(n))$ if $f = g \cdot o(1)$, where $o(1) \underset{n \to \infty}{\longrightarrow} 0$.

 We will say that a property holds \emph{asymptotically almost surely} (denoted a.a.s.) for the random graph if the probability that it holds for $G(n, p)$ goes to $1$ as $n$ goes to infinity. We denote probability with $\Prob$ and expected value with $\E$.  In this paper we will compute the edge dimension of $G(n, p)$, where $G(n, p)$ is the Erd\H{o}s-R\'enyi random graph on $n$ vertices with edge probability $p$, so that $\E\left[|E|\right] = np$. 
\section{The upper bound}
In this section we prove the following theorem: 
\begin{theorem}\label{upperbound}
$$\edim(G(n, p)) \leq (1 + o(1))\frac{4\log(n)}{\log(1/q)}, $$
where $q = 1 - 2p(1-p)^2(2-p)$.
\end{theorem}

Before proving Theorem \ref{upperbound} we need some lemmas.

\begin{lemma}
Fix $\omega \in \{ 1, \ldots , n \} $.
Suppose for any two distinct uniformly random edges $e_1, e_2$ the following property holds: for a uniformly random subset $W \subseteq V$ of size $\omega$, 
$$ \Prob(W \text{does not distinguish } e_1, e_2) \leq 1/{n^4p^2}.$$ 
Then: $$\edim(G) \leq \omega.$$
\end{lemma}

\begin{proof}
We use the probabilistic method. Indeed, $$\E\left[\abs{E(G(n, p))}\right] = p {n \choose 2} < pn^2/2,$$ so  the expected number of distinct pairs of edges is no more than $ {pn^2/2 \choose 2} \leq p^2n^4/8$. Then by our hypothesis the expected number of pairs not distinguished by some $W \subseteq V$ with $|W| = \omega$ is less than $ p^2n^4/8p^2n^4 = 1/8$. Since this is strictly less than $1$, there must be at least one such set $W$ that distinguishes all the pairs. 
\end{proof}

 \begin{lemma}\label{prob v not dist}
 In   $G(n, p)$, the probability that a vertex $v$ doesn't distinguish two uniformly random edges $e_1, e_2$ is $(1 + o(1))q$, where $q = 1 - 2p(1-p)^2(2-p)$.
 \end{lemma}
\begin{proof}
There are two types of distinct edge pairs: 

1. $ab, bc$ for some $a, b, c \in V$.

2. $ab, cd$ for $a, b, c, d \in V$ and $\{a, b \} \bigcap \{c, d \} = \emptyset$.\\
Note that 
$$\text{the expected number of type $2$ pairs} = 3{n \choose 4}p^2 = \frac{n^4p^2}{8}(1 + o(1)), $$
and 
$$\text{the expected number of type $1$ pairs} \leq n^3 = o\left(\frac{n^4p^2}{8}\right). $$
 Thus, we can neglect the type $1$ pairs. Let $xy, zt$ be a type 2 pair and $v$ a uniformly random vertex. Clearly, $\Prob(v \in \{ x, y, z, t \} ) = o\left(\frac{n^4p^2}{8}\right)$, so we can assume $v$ is not a vertex of $xy$ or $zt$. Since the random graph has diameter $2$ a.a.s. (see \cite{moon1966matrix}), $v$ has distance $1$ or $2$ to $x, y, z, t$ a.a.s.; moreover, $\Prob(d(v, x) = 1) = p$, so a.a.s. $\Prob(d(v, x) = 2) = 1-p$. It is easy to see that $v$ has distance $1$ to $xy$ and $2$ to $zt$ if and only if one of the following cases holds:
\begin{align*}
&1.\ (d(v, x), d(v, y), d(v, z), d(v, t)) = (1, 1, 2, 2) \text{ (with probability }p^2(1-p)^2).\\
&2.\ (d(v, x), d(v, y), d(v, z), d(v, t)) = (1, 2, 2, 2) \text{ (with probability }p(1-p)^3).\\
&3.\ (d(v, x), d(v, y), d(v, z), d(v, t)) = (2, 1, 2, 2)\text{ (with probability }p(1-p)^3).\\
\end{align*}
The same probabilities hold for $xy$ and $zt$ switched. Thus, a.a.s.
\begin{align*}
\Prob(v \text{ distinguishes } xy, zt) &= (1 + o(1))\cdot 2(p^2(1-p)^2 + 2p(1-p)^3) \\
&= (1  + o(1))\cdot 2p(1-p)^2(2-p)= (1+o(1))(1-q). 
\end{align*} This gives us the desired result.
\end{proof} 

\begin{lemma}\label{prob W not dist}
Consider a uniformly random subset $W \subseteq V(G(n, p))$ with $$|W| = (1 + o(1))\frac{4\log(n)}{\log(1/q)}.$$ Then for uniformly random $e_1$ and $e_2 \in E(G(n, p))$, $$\Prob(W\text{ does not distinguish } e_1, e_2) \leq (1 + o(1))/n^4p^2.$$ 
\end{lemma}
\begin{proof}
Using Lemma \ref{prob v not dist}, we see that
 \begin{align*}
&\Prob(W \text{doesn't distinguish }e_1, e_2)  \\
\leq &(1 + o(1))\Prob(\text{uniformly random}\text{ vertex }v  \text{ doesn't distinguish }e_1, e_2)^{|W|}  \\  
\leq &(1 + o(1)) q^{(1+o(1))\frac{4\log(n)}{\log(1/q)}} \\ 
 = &(1+o(1))q^{-\log_q(n^4)}  \\ 
= &  (1+o(1))\frac{1}{n^4}  \\
 \leq &(1+o(1)) \frac{1}{p^2n^4}.
\end{align*}

\end{proof}

\begin{proof}[Proof of Theorem \ref{upperbound}]
Combining Lemma \ref{prob W not dist} and 3.2, we see that an edge basis of $G(n, p)$ has cardinality at most 
$$(1 + o(1))\frac{4\log(n)}{\log(1/q)},$$
 which concludes the proof of Theorem \ref{upperbound}.
\end{proof}

\section{The lower bound}
The goal of this section is to prove the following theorem: 
\begin{theorem}\label{lower bound}
For the random graph $G(n, p)$ we have 

$$\edim(G(n, p)) \geq (1 + o(1))\frac{4\log(n)}{\log(1/q)}, $$
where $q = 1 - 2p(1-p)^2(2-p)$.
\end{theorem}

Let $$\eps := \frac{3\log\log(n)}{\log(n)} = o(1).$$
We will show that a.a.s. there is no edge generating set $R$ of cardinality less than
\begin{gather*}
r:= \frac{(4-\eps)\log(n)}{\log(1/q)}.
\end{gather*}
To do that we will use a theorem which is a version of Suen's inequality provided by Janson in \cite{janson1998new}. First we introduce some notation:
\begin{itemize}
\item  $\{I_i\}_{i \in \mathcal{I}}$ --- a finite family of indicator random variables;
\item  $\Gamma$ --- the associated dependency graph ($\mathcal{I}$ is the set of vertices of $\Gamma$);
\item For $i, j \in \mathcal{I}, \text{ write }i \sim j$ if $i, j$ are adjacent in $\Gamma$;
\item $ \mu: = \sum_i \Prob(I_i = 1)$
\item $\Delta: = \sum_{i \sim j} \E[I_i I_j]$
\item $\delta: = \max_i{\sum\limits_{i \sim j} \Prob(I_j)}$
\item $S := \sum\limits_{i} I_i$
\end{itemize}
\begin{theorem}[Suen's inequality: Theorem 2 of \cite{janson1998new}]
$$\Prob(S = 0) \leq \exp( - \mu + \Delta\eps^{2\delta})$$
\end{theorem}

We now apply this theorem to our problem.

Fix $R \subseteq V(G(n, p))$ with $|R| = r$. 
Let $$\mathcal{I} := \{ (xy, zt)\vert xy, zt  \in E(G(n,p)), xy \neq zt\}$$ be the set of pairs of pairs of distinct edges, and for any $(xy, zt) \in \mathcal{I}$ let $A_{xy, zt}$ to be the event $d_R(xy) = d_R(zt)$ (with $I_{xy, zt}$ being the corresponding indicator function). Let $S = \sum_{(xy, zt) \in \mathcal{I}}I_{xy, zt}$. Then $$\Prob(R\text{ is an edge generating set}) = \Prob(S = 0).$$
The associated dependency graph has $\mathcal{I}$ as vertices and $(x_1y_1, z_1t_1) \sim  (x_2y_2, z_2t_2)$ if and only if $\{ x_1, y_1, z_1, t_1 \} \cap \{ x_2, y_2, z_2, t_2 \} \neq \emptyset $ (here, again, $\sim$ denotes adjacency). Then by Theorem 4.2, 
\begin{gather}
\Prob(S = 0) \leq \exp( - \mu + \Delta\eps^{2\delta}), 
\end{gather}
where 
\begin{gather*}
\mu := \sum_{(e, f) \in \mathcal{I}} \Prob(A_{e, f}),\\
\Delta: = \sum_{(e_1,f_1) \sim (e_2, f_2)} \E[I_{e_1f_1} I_{e_2f_2}],\\
\delta := \underset{(e_1, f_1) \in \mathcal{I}}{\max} \sum_{(e_2, f_2) \sim (e_1, f_1)} \Prob(A_{e_2, f_2}).
\end{gather*}

We now show the following estimate for $\mu$:
\begin{lemma}[Evaluation of $\mu$]
\begin{gather*}
\mu =(1 + o(1))p^2 n^{\eps} /8.
\end{gather*}
\end{lemma}
\begin{proof}
Using Lemma \ref{prob v not dist}, we can derive that that 
$$\Prob(A_{e, f}) = (1 + o(1))q^r,$$
so, since the expected number of pairs is $(1+o(1))(n^4p^2/8)$, we indeed get
$$\mu = (1 + o(1))n^4p^2 q^r/8.$$
Since 
$r = \frac{(4-\eps)\log(n)}{\log(1/q)}$, 
\begin{gather}
q^r = q^{-(4-\eps)\log_q(n)} = n^{\eps-4}. 
\end{gather}
Thus,
 $$(1 + o(1))n^4p^2q^r/8 = (1 + o(1))n^4p^2 n^{\eps-4} /8  =  (1 + o(1))p^2 n^{\eps} /8.$$
This means that, indeed, 
$$\mu =(1 + o(1))p^2 n^{\eps} /8.$$ 
\end{proof}
Now we estimate $\Delta$ and show the following:
\begin{lemma}[Evaluation of $\Delta$]
\begin{gather*}
\Delta = o(\mu).
\end{gather*}
\end{lemma}
\begin{proof}

\begin{claim} In calculating $\Delta$, we may only consider the adjacent pairs  $$ (x_1y_1, z_1t_1), (x_2y_2, z_2t_2) \in \mathcal{I}$$ for which $$|\{ x_1, y_1, z_1, t_1 \} \cap \{ x_2, y_2, z_2, t_2 \}| = 1.$$
\end{claim}
\begin{proof}[Proof of claim]
Consider two adjacent elements of $\mathcal{I}: (x_1y_1, z_1t_1) \sim (x_2y_2, z_2t_2)$. Suppose $| \{x_1, y_1, z_1, t_1, x_2, y_2, z_2, t_2 \}| = 7$. The expected number of such pairs is $$p^4 \frac{n!}{16\cdot(n-7)!} = (1+o(1))p^4 n^7/16.$$ Now consider two adjacent elements of  $\mathcal{I}$ with $|\{ x_1, y_1, z_1, t_1, x_2, y_2, z_2, t_2 \}| \leq 6.$ There are no more than $$n^6 = o(p^4n^7/16)$$ such pairs of pairs. 
\end{proof}

Thus we can and will only consider pairs of elements of $\mathcal{I}$  with only one vertex in common. 

We will now compute the probability that $I_{(x_1y_1, z_1t_1)}I_{(x_1y_2, z_2t_2)} = 1$.
Consider a uniformly random vertex $v$. We can neglect the case when $v \in \{ x_1$, $y_1$, $z_1$, $t_1$, $y_2$, $z_2$, $t_2 \}$ because it happens with probability $o(1)$. Since the random graph has diameter $2$ a.a.s., $I_{(x_1y_1, z_1t_1)}I_{(x_1y_2, z_2t_2)} = 1$ in the following cases:

Case 1: $d_v(x_1) = 1$. Then $v$ has to have distance $1$ to all four edges. $v$ has distance $1$ to $z_1t_1$ (or $z_2t_2$) with probability $p^2 + 2p(1-p) = p(2-p)$, and the distances from $v$ to $y_1, y_2$ don't affect anything, so
 $$\Prob\left( I_{(x_1y_1, z_1t_1)}I_{(x_1y_2, z_2t_2)} = 1 \vert \text{ case 1 holds} \right) = p^3(2-p)^2.$$

Case 2: $d_v(x_1) = 2$. Then $v$ has distance $2$ to both $x_1y_1$ and $z_1t_1$ with probability $(1-p)^3$ and distance $1$ to both $x_1y_1$ and $z_1t_1$ with probability $p^2(2-p)$. So $v$ is equidistant from the two edges with probability $ (1-p)^3 + p^2(2-p)$. Thus, 
$$\Prob\left( I_{(x_1y_1, z_1t_1)}I_{(x_1y_2, z_2t_2)} = 1 \vert \text{ case 2 holds} \right)  = (1-p)((1-p)^3 + p^2(2-p))^2.$$ 

Hence the total probability  $$\Prob\left(I_{(x_1y_1, z_1t_1)}I_{(x_1y_2, z_2t_2)} = 1\right) = (1-p)((1-p)^3 + p^2(2-p))^2 + p^3(2-p)^2 .$$ We will henceforth refer to this constant as $s_p$: $$s_p := (1-p)((1-p)^3 + p^2(2-p))^2 + p^3(2-p)^2. $$ 
It follows that
$$\Delta = (1 + o(1)) p^4 n^7 s_p^r /16.$$ Using (2), we get 
\begin{align*}
\Delta &= (1 + o(1)) p^4 n^7 s_p^r /16 \\
&= (1 + o(1)) p^4n^3 n^{\eps}n^{4-\eps} s_p^r /16 \\
 &= (1 + o(1))\frac{p^2}{2} \left(\frac{s_p}{q}\right)^r \frac{p^2 n^{\eps}}{8} \\
& = (1+o(1))\frac{p^2 }{2}\left(\frac{s_p}{q}\right)^r \mu. 
\end{align*}
Notice that
\begin{align*}
\left(\frac{s_p}{q}\right)^r &=\left(\frac{s_p}{q}\right)^{(4 - \eps)\log(n))/\log(q)} \\
 &= n^{(4 - \eps)\log\left(\frac{s_p}{q}\right)/\log(1/q)}\\
 &= n^{(4 - \eps)\left(\frac{-\log(s_p)}{\log(q)} - 1\right)} \\ 
&= n^{(4 - \eps)(-\log_q s_p - 1)} \leq  n^{\eps-4}.
\end{align*}
Thus, $$  (1+o(1))\frac{p^2 }{2}\left(\frac{s_p}{q}\right)^r \mu \leq (1+o(1))\frac{p^2  n^{\eps - 4}}{2}\mu = o(\mu).$$ This concludes the proof that 
$$\Delta = o(\mu).$$
\end{proof}

Finally, we estimate $\delta$ and show the following:
\begin{lemma}[Evaluation of $\delta$]
\begin{gather*}
\delta = o(1).
\end{gather*}
\end{lemma}
\begin{proof}
Note that for fixed $f_1, e_1$, 
\begin{gather*}
\Prob \left(A_{e_2, f_2} \vert \ (e_2, f_2) \text{ uniformly random}, (e_2, f_2) \sim (e_1, f_1)\right) = \\
\Prob(A_{e, f} \vert \ e, f  \text{ uniformly random}).
\end{gather*}
 Thus, the maximum for $\delta$ is achieved for $(e_1, f_1)$ with the largest possible number of adjacent edge pairs $(e_2, f_2)$. Clearly, this number is the greatest when $e_1$ and $f_1$ don't share vertices. The expected number of adjacent edge pairs in this case is $(1 + o(1))4p^2q^rn^3/2 = (1 + o(1))2n^3p^2$. Since $q^r =\Prob(A_{e, f})$ for uniformly random edges $e, f$ we have 
$$2\delta = (1 + o(1))2n^3p^2q^r.$$
Using (2), we get 
$$ \delta=  (1 + o(1))2n^3p^2n^{\eps - 1} = o(1).$$
\end{proof}
We are now ready to prove Theorem \ref{lower bound}.
\begin{proof}[Proof of Theorem \ref{lower bound}]
Substituting the results of Lemma 4.3, 4.4, 4.5  into inequality (1), we obtain 
\begin{align*}
\log\left(\Prob(S=0)\right) &\leq (1+o(1))\left(-\mu + o(\mu)e^{o(1)}\right) \\
&\leq (1 + o(1))\left( -\mu + o(\mu)\right) \\
&\leq -(1+o(1))\mu \\
& \leq -(1 + o(1))p^2n^\eps/8 \\
&\leq -p^2n^\eps/16 
\end{align*}
for sufficiently large $n$.
Then the expected number of edge generating sets of cardinality $r$ is no greater than 
\begin{align*}
{n \choose r}\exp(-p^2n^\eps/16) &\leq n^r \exp(p^2n^\eps/16) \\
&= \mathcal{O}\left(\exp[(4-\eps)\log^2(n)/\log(1/q) - p^2n^\eps/16]\right) \\
&\leq \mathcal{O}\left(\exp[\log^2(n)- \log^3(n)p^2/16]\right) \\
&= o(1).
\end{align*}
This concludes the proof of Theorem \ref{lower bound}.
\end{proof}

\section{Concluding remarks}
We have shown that $$\edim(G(n, p)) = (1 + o(1))\frac{4\log(n)}{\log(1/q)},$$ where $$q = 1 - 2p(1-p)^2(2-p).$$ As demonstrated by Bollobas in \cite{bollobas2012metric}, $$\dim(G(n, p)) = (1 + o(1))\frac{2\log(n)}{\log(1/Q)},$$ where $Q = p^2 + (1-p)^2$. Since $2/\log(1/Q) < 4/\log(1/q)$, this means that
\begin{gather*}
\dim(G(n, p)) <\edim(G(n, p))
\end{gather*}
a.a.s. for all $p \in (0, 1)$.

While random graphs with constant edge probability don't help in resolving the problem of finding more examples of graphs $G$ for which $\edim(G) < \dim(G)$ posed in \cite{2016arXiv160200291K}, perhaps this problem could be addressed with random graphs of non-constant probability $p(n)$. Because of this it would be interesting to calculate  $\edim(G(n, p(n))$ for non-constant $p(n)$. As mentioned earlier, relevant results for $\dim(G(n, p(n)))$ can be found in \cite{bollobas2012metric}.

\section{Acknowledgments}
The research was conducted during the Undergraduate Mathematics Research Program at University of Minnesota Duluth, and supported by grants NSF-1358659 and NSA H98230-16-1-0026. 
I would like to thank Joe Gallian very much for creating a marvelous working environment and all his invaluable support throughout the program. I'm very thankful to thank Matthew Brennan for suggesting the problem; I would also like to thank Levent Alpoge for pointing out the paper \cite{bollobas2012metric}, encouraging me to work on this problem, extremely helpful commentary and remarks, and everlasting patience. 

\bibliographystyle{ieeetr}
\bibliography{RandGraphEdim}{}

\begin{thebibliography}{10}

\bibitem{bollobas2012metric}
B.~Bollob{\'a}s, D.~Mitsche, and P.~Pralat, ``Metric dimension for random
  graphs,'' {\em arXiv:1208.3801}, 2012.

\bibitem{2016arXiv160200291K}
A.~{Kelenc}, N.~{Tratnik}, and I.~G. {Yero}, ``{Uniquely identifying the edges
  of a graph: the edge metric dimension},'' {\em ArXiv e-prints}, Jan. 2016.

\bibitem{slater1975leaves}
P.~J. Slater, ``Leaves of trees,'' {\em Congr. Numer.}, vol.~14, no.~549-559,
  p.~37, 1975.

\bibitem{harary1976metric}
F.~Harary and R.~A. Melter, ``On the metric dimension of a graph,'' {\em Ars
  Combin.}, vol.~2, no.~191-195, p.~1, 1976.

\bibitem{khuller1996landmarks}
S.~Khuller, B.~Raghavachari, and A.~Rosenfeld, ``Landmarks in graphs,'' {\em
  Discrete Appl. Math.}, vol.~70, no.~3, pp.~217--229, 1996.

\bibitem{chartrand2000resolvability}
G.~Chartrand, L.~Eroh, M.~A. Johnson, and O.~R. Oellermann, ``Resolvability in
  graphs and the metric dimension of a graph,'' {\em Discrete Appl. Math.},
  vol.~105, no.~1, pp.~99--113, 2000.

\bibitem{melter1984metric}
R.~A. Melter and I.~Tomescu, ``Metric bases in digital geometry,'' {\em
  Computer Vision, Graphics, and Image Processing}, vol.~25, no.~1,
  pp.~113--121, 1984.

\bibitem{chartrand2000resolvability1}
G.~Chartrand, C.~Poisson, and P.~Zhang, ``Resolvability and the upper dimension
  of graphs,'' {\em Comput. Math. with Appl.}, vol.~39, no.~12, pp.~19--28,
  2000.

\bibitem{johnson1993structure}
M.~Johnson, ``Structure-activity maps for visualizing the graph variables
  arising in drug design,'' {\em Biopharmaceutical Stat.}, vol.~3, no.~2,
  pp.~203--236, 1993.

\bibitem{moon1966matrix}
J.~W. Moon and L.~Moser, ``A matrix reduction problem,'' {\em Mathematics of
  Computation}, vol.~20, no.~94, pp.~328--330, 1966.

\bibitem{janson1998new}
S.~Janson, ``New versions of suen's correlation inequality,'' {\em Random
  Struct. and Algor.}, vol.~13, no.~3-4, pp.~467--483, 1998.

\end{thebibliography}

\end{document}